\documentclass[a4paper,11pt]{amsart}
\usepackage{graphicx}
\usepackage{amssymb}
\usepackage{amsmath}
\usepackage{latexsym}
\usepackage{amsthm}
\usepackage{autograph,eepic,epic,latexsym,bezier,amsbsy,color,enumerate,amsfonts,amsmath,amscd,amssymb}
\usepackage[latin1]{inputenc}
\usepackage{rotating}
\setloopdiam{7}\setprofcurve{7}
\newtheorem{defi}{Definition}[section]
\newtheorem{teo}[defi]{Theorem}
\newtheorem{lem}[defi]{Lemma}
\newtheorem{prop}[defi]{Proposition}
\newtheorem{cor}[defi]{Corollary}

\newtheorem{os}[defi]{Remark}

\newcommand{\B}{\operatorname{B}}

\def\subjclass#1{{\renewcommand{\thefootnote}{}%
\footnote{\emph{Mathematics Subject Classification (2010):} #1}}}

\begin{document}

\title[The Tutte Polynomial of the Grigorchuk group and the Basilica group]{The Tutte Polynomial of the Schreier graphs of the Grigorchuk group and the Basilica group}

\author[T. Ceccherini-Silberstein]{Tullio Ceccherini-Silberstein}
\address{%
Tullio Ceccherini-Silberstein\\
Dipartimento di Ingegneria\\
Università del Sannio\\
C.so Garibaldi, 107\\
82100 Benevento, Italia}\email{tceccher@mat.uniroma1.it}

\author[A. Donno]{Alfredo Donno}
\address{%
Alfredo Donno\\
Dipartimento di Matematica \lq\lq G. Castelnuovo\rq\rq\\
Sapienza Università di Roma\\
Piazzale A. Moro, 2\\
00185 Roma, Italia}\email{donno@mat.uniroma1.it}

\author[D. Iacono]{Donatella Iacono}
\address{%
Donatella Iacono\\
Dipartimento di Matematica \lq\lq G. Castelnuovo\rq\rq\\
Sapienza Università di Roma\\
Piazzale A. Moro, 2\\
00185 Roma, Italia}\email{iacono@mat.uniroma1.it}

\date{July 17, 2010}

\begin{abstract}
We study the Tutte polynomial of two infinite families of finite
graphs. These are the Schreier graphs associated with the action
of two well-known self-similar groups acting on the binary rooted
tree by automorphisms: the first Grigorchuk group of
intermediate growth, and the iterated monodromy group of the
complex polynomial $z^2-1$ known as the  Basilica
group. For both of them, we describe the Tutte polynomial
and we compute several special evaluations of it, giving
further information about the combinatorial structure of these
graphs.
\end{abstract}



\maketitle

\section{Introduction}

The Tutte polynomial is a two-variable polynomial which can be
associated with a graph, a matrix, or, more generally, with a
matroid. It has many interesting applications in several areas of
sciences as, for instance, Combinatorics, Probability, Statistical
Mechanics, Computer Science and Biology. It was introduced by W.T.
Tutte \cite{tutte47, tutte54, tutte67} and we will mainly refer to
\cite{bollobas, equivalence, machi, tutte04} as expository papers.

Given a finite graph $G$, its Tutte polynomial $T(G;x,y)$
satisfies a fundamental universal property with respect to the
deletion-contraction reduction of the graph. Hence, any
multiplicative graph invariant with respect to a
deletion-contraction reduction turns out to be an evaluation of
it. This polynomial is quite interesting since several
combinatorial, enumerative and algebraic properties of the graph
-- such as the number of spanning trees,
of spanning connected subgraphs, of spanning forests and of acyclic
orientations of the graph --
can be investigated by considering special evaluations of this polynomial. 
Moreover, from the Tutte polynomial one also recovers the reliability and the chromatic polynomials. It has also many interesting connections with statistical
mechanical models as the Potts model \cite{potts}, percolation
\cite{percolation}, the Abelian Sandpile Model \cite{cori,
merino}, as well as with the theory of error correcting codes
\cite{potts}.

In this paper, we study the Tutte polynomial of the Schreier
graphs associated with the action of two well-known automorphism
groups of the binary rooted tree: the Grigorchuk group and the 
iterated monodromy group of the complex polynomial $z^2-1$ known as the Basilica
group. See also \cite{donnoiacono}, where the Tutte polynomials of
the Sierpi\'{n}ski graphs  and the Schreier graphs
of the Hanoi Towers group $H^{(3)}$  are computed.

The first Grigorchuk was introduced by R. Grigorchuk in 1980; it
yields the simplest solution of the Burnside problem (an infinite,
finitely generated torsion group) and the first example of a finitely generated group
of intermediate (i.e.  faster than polynomial but slower than exponential) growth. 
See \cite{CMS} and \cite{grigorchuk} for a detailed account and further references.

The Basilica group was introduced by R. Grigorchuk and A. \.{Z}uk
in \cite{primo} as a group generated by a three-state automaton.
It is a remarkable fact due to V. Nekrashevych \cite{volo} that
this group can be described as the iterated monodromy group
 of the complex polynomial $z^2-1$; therefore, there exists
 a natural way to associate with it a compact limit space
homeomorphic to the well-known Basilica fractal. Moreover, it is
the first example of an amenable group (a highly non--trivial and
deep result of L. Bartholdi and B. Vir\'ag \cite{amen}) not
belonging to the class of subexponentially amenable groups, which
is the smallest class containing all groups of subexponential
growth and closed after taking subgroups, quotients, extensions
and direct unions.

Over the last decade, Grigorchuk and a number of coauthors have
developed a new exciting direction of research focusing on
finitely generated groups acting by automorphisms on rooted trees,
transitively on each level \cite{fractal}. They proved that these
groups have deep connections with the theory of profinite groups
and with complex dynamics. In particular, many groups of this type
satisfy a property of self-similarity (see Definition
\ref{defiselfsimilar}), reflected on fractalness of some
limit objects associated with them \cite{volo}.

In Sections \ref{sezione Grigor} and \ref{sezio Basilica}, we
study the Tutte polynomial for the Schreier graphs
$\{\Gamma_n\}_{n\geq 1}$ and $\{\B_n\}_{n\geq 1}$ of the
Grigorchuck group and the Basilica group, respectively. It follows
from the recursive expression of the generators of these groups
that these graphs have a cactus structure, i.e., they are union of
cycles, arranged in a tree-like way. This enables us to compute the
Tutte polynomial using the multiplicative property
\eqref{prodotto} (see Section \ref{subsec.preli TUTTE}). Once we
have these polynomials, we compute many special
evaluations of them, providing several interesting
information about the combinatorial structure of these graphs and
showing connections with reliability, colorability and the Ising model
(see Section \ref{subsec.preli TUTTE} for definitions and
details). Note that some evaluations of the Tutte polynomial are
trivial, when the graphs $\{\Gamma_n\}_{n\geq 1}$ and
$\{\B_n\}_{n\geq 1}$ are considered with loops; for instance, the number of acyclic orientations, the chromatic polynomial and the partition function of the Ising model.
Therefore, we make these computations on the graphs $\{\Gamma_n^\ast \}_{n\geq 1}$ and $\{\B_n^\ast\}_{n\geq 1}$, obtained from the graphs
$\{\Gamma_n\}_{n\geq 1}$ and $\{\B_n\}_{n\geq 1}$, respectively,
by deleting loops. For the considered graphs, we explicitly describe:
\begin{itemize}
\item {\it the Tutte polynomial} (Theorems \ref{thmtuttegrigorchuk} and
\ref{thmtuttebasilica});\smallskip
\item {\it the reliability polynomial} (Propositions \ref{prop.reliabGRIG} and
\ref{prop.reliability BAS});\smallskip
\item {\it the complexity} (Propositions
\ref{pro.complexi.GRIG} and \ref{prop.complexity BAS});\smallskip
\item {\it the number of connected spanning subgraphs} (Propositions \ref{prop.conn.span.subgra.GRIG} and
\ref{prop.conn.spanning.subgr BAS});\smallskip
\item {\it the number of spanning forests} (Propositions
\ref{prop.conn.span.forest.GRIG} and \ref{prop.spanning forest
BAS});\smallskip
\item {\it the number of acyclic orientations} (Propositions \ref{prop.acyclic.orien.GRIG} and
\ref{prop.acyclic orienta BAS});\smallskip
\item {\it the chromatic polynomial} (Propositions
\ref{prop.chromatic poly.GRIG} and \ref{prop.chromatic polyno
BAS});\smallskip
\item {\it the partition function of the Ising model} (Theorems \ref{thmisinggrig} and \ref{thmisingBas}).
\end{itemize}

\section{Preliminaries}\label{Section preliminare}

\subsection{The Tutte polynomial} \label{subsec.preli TUTTE}
 Throughout the paper, we deal with graphs which are connected and
finite. Moreover, both multiple edges and multiple loops are
allowed. As usual, $G=(V(G),E(G))$ denotes a graph with vertex set
$V(G)$ and edge set $E(G)$; we will often write $V$ and $E$, when
there is no risk of confusion, and so $G=(V,E)$. Moreover, we
denote by $E_n$ the graph with $n$ vertices and no edges, and by
$K_n$ the complete graph on $n$ vertices. A subgraph
$A=(V(A),E(A))$ of a graph $G=(V(G),E(G))$ is said to be \emph{spanning}
if the condition $V(A)=V(G)$ is satisfied. In particular, a
\emph{spanning subtree} of $G$ is a spanning subgraph of $G$ which
is a tree, a \emph{spanning forest} of $G$ is a spanning subgraph
of $G$ which is a forest. The number of spanning trees of a graph
$G$ is called \emph{complexity} of $G$ and is denoted by
$\tau(G)$.  It is interesting to study complexity when the system
grows. More precisely, given a sequence $\{G_n\}_{n\geq 1}$ of
finite graphs with complexity $\tau(G_n)$, such that $|V(G_n)|\to
\infty$, the limit
$$
\lim_{|V(G_n)| \to \infty}\frac{\log \tau(G_n)}{|V(G_n)|},
$$
when it exists, is called the \emph{asymptotic growth
constant} of the spanning trees of $\{G_n\}_{n\geq 1}$ \cite{lyons}.\\
Finally, let $k(G)$ be the number of connected components of $G$.

\begin{defi}
Let $A$ be a spanning subgraph of $G$, then the rank $r(A)$ and
the nullity $n(A)$ of $A$ are defined as
\[
r(A)=|V(A)| -k(A)= |V(G)|-k(A) \quad \mbox{ and } \quad
n(A)=|E(A)| - r(A)=|E(A)|-|V(A)| + k(A).
\]
\end{defi}

\begin{defi}[Spanning subgraphs]\label{defspanning}
Let $G=(V,E)$ be a graph. The Tutte polynomial $T(G;x,y)$ of $G$
is defined as
\begin{eqnarray}\label{defsubgraphs}
T(G;x,y)= \sum_{A\subseteq G}(x-1)^{r(G)-r(A)}(y-1)^{n(A)},
\end{eqnarray}
where the sum runs over all the spanning subgraphs $A$ of $G$.
\end{defi}

The Tutte polynomial can be also defined by a recursion process
given by deleting and contracting edges. We recall that, given
$G=(V,E)$, the graph $G\setminus e=(V,E-\{e\})$ is obtained from
$G$ by deleting the edge $e\in E$. The graph obtained by
contracting an edge $e\in E$ is the result of the identification
of the endpoints of $e$ followed by removing $e$. We denote it by
$G/ e$. Finally, we recall that an edge in a connected graph is a
\emph{bridge} if its deletion disconnects the graph, it is a
\emph{loop} if its endpoints coincide.

\begin{defi}[Deletion-Contraction]\label{contracting}
Let $G=(V,E)$ be a graph. The Tutte polynomial $T(G;x,y)$ of $G$
is defined as
$$T(G;x,y)=
\left\{
\begin{array}{ll}
   1  & \hbox{if $G=E_1$;} \\
   xT(G\setminus e;x,y)  & \hbox{if e is a bridge;} \\
   yT(G\setminus e;x,y)  & \hbox{if e is a loop;} \\
   T(G\setminus e;x,y) + T(G/ e;x,y)  & \hbox{if e is neither a bridge nor a loop.} \\
  \end{array}
\right.$$
\end{defi}

The recursive process to compute the Tutte polynomial in this
second definition is independent on the order in which the edges
are chosen: this can be proven by showing that Definitions
\ref{defspanning} and \ref{contracting} are equivalent
\cite{equivalence}.

Once we have the definition, we can state some of the main
properties of the Tutte polynomial (for more details, see
\cite{bollobas, equivalence, machi}). Recall that a one point join
$G*H$ of two graphs $G$ and $H$ is obtained by identifying a
vertex $v$ of $G$ and a vertex $w$ of $H$ into a single vertex of
$G*H$.  The following property can be easily
proven by using Definition \ref{defspanning}:
\begin{eqnarray}\label{prodotto}
T(G*H;x,y)=T(G;x,y)T(H;x,y).
\end{eqnarray}
This property will be fundamental for our computations and we will
refer to it as the \emph{multiplicative property} of the Tutte
polynomial. The next lemma follows from Definition
\ref{contracting}, by using the multiplicative property.

\begin{lem}\label{singlecycle}
If $C_k$ is a cycle of length $k$, with $k\geq 2$, then its Tutte
polynomial $T(C_k;x,y)$ is $ y+x+x^2+\cdots +x^{k-1}$.
\end{lem}

\begin{proof}
The proof is by induction on the length $k$ of the cycle. For
$k=2$, $C_2$ is a $2$-cycle. It is trivial to verify that the
Tutte polynomial of the graph
\begin{center}
\begin{picture}(300,8)
\letvertex A=(130,4)\letvertex B=(170,4)
\drawvertex(A){$\bullet$}\drawvertex(B){$\bullet$}
\drawundirectededge(A,B){}
\end{picture}
\end{center}
is $x$, using Definition \ref{contracting}. Therefore, for the
graph $C_2$
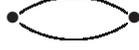
\begin{center}
\begin{picture}(300,8)
\letvertex A=(130,4)\letvertex B=(170,4)
\drawvertex(A){$\bullet$}\drawvertex(B){$\bullet$}
\drawundirectedcurvededge(A,B){} \drawundirectedcurvededge(B,A){}
\end{picture}
\end{center}
the Tutte polynomial is $T(C_2;x,y)=y+x$, again by Definition
\ref{contracting}, so that the assertion is true for $k=2$. Now
let $C_{k+1}$ be a cycle of length $k+1$. Let $e$ be a fixed edge
of $C_{k+1}$, so that by Definition \ref{contracting},
$$
T(C_{k+1};x,y)= T(C_k;x,y) + T(C_{k+1}\setminus e;x,y).
$$
Since, by the multiplicative property, $T(C_{k+1}\setminus e;x,y) =
x^k$, we can apply induction and we get the required result.
\end{proof}

In the next sections, we will be interested in special evaluations
of the Tutte polynomial, that allow us to deduce many
combinatorial and algebraic properties of the graphs considered.
In the following theorem, we collect many of these properties that
are well-known in literature.

\begin{teo} \cite[Theorem 3 and 8]{machi}\label{evaluations}
Let $G=(V,E)$ be a connected graph and denote by $T(G; x,y)$ its
Tutte polynomial. Then:
\begin{enumerate}
\item $T(G; 1,1) =\tau(G)$;
\item $T(G; 1,2)$ is the number of spanning connected subgraphs of $G$;
\item $T(G; 2,1)$ is the number of spanning forests  of $G$;
\item $T(G; 2,2)= 2^{|E|}$;
\item $T(G; 2,0)$ is the number of acyclic orientations of $G$, i.e., orientations having no oriented
cycles.
\end{enumerate}
\end{teo}
Another fundamental interesting aspect of the Tutte polynomial is
that, starting from it, one can obtain other interesting
polynomials associated with the graph: the \emph{reliability
polynomial} and the \emph{chromatic polynomial}.\\ \indent More
precisely, as regards the reliability polynomial $R(G,p)$, suppose
that each edge of $G$ is independently chosen to be active (or
open) with probability $p$ or inactive (closed) with probability
$1-p$. Then, $R(G,p)$ is defined as the probability that in this
random model there is a
path of active edges between each pair of vertices of $G$.\\
\indent As regards the colorability, we recall that a proper (or
admissible) $\lambda$-coloring of the vertices of $G$ is an
assignment of $\lambda$ colors to the vertices of $G$, in such a
way that adjacent vertices have distinct colors. $G$ is said
$\lambda$-colorable if it admits a proper $\lambda$-coloring. The
\emph{chromatic number} $\chi(G)$ of $G$ is defined as the minimal
number $\lambda$ such that $G$ is $\lambda$-colorable. $G$ is
uniquely $\lambda$-colorable if $\chi(G)=\lambda$ and any
$\lambda$-coloring of $G$ induces the same partition of $V(G)$
(vertices with the same color are in the same class). The
chromatic polynomial $\chi(G, \lambda)$ gives, for all values
$\lambda$, the number of proper $\lambda$-colorings of $G$. The
famous $4$-color Theorem states that, if $G$ is a planar graph,
then $\chi(G,4)>0$.

\noindent The connection with the Tutte polynomial is given by the
following theorem.
\begin{teo}\cite[Theorem 12 and 17]{machi}\label{twopolynomials}
 Let $G=(V,E)$ be a graph. Then,
\begin{enumerate}
\item $\displaystyle
R(G,p)=p^{|V(G)|-1}(1-p)^{|E(G)|-|V(G)|+1}T\left(G;1,\frac{1}{1-p}\right)$;
\item $\chi(G, \lambda) = (-1)^{r(G)}\lambda ^{k(G)}
T(G;1-\lambda,0)$.
\end{enumerate}
\end{teo}

Finally, we want to recall the well-known connection between the
Tutte polynomial of a graph and the Ising model on it, which is
obtained as a special case of the $Q$-Potts model, for $Q=2$.\\
\indent The famous Ising model of ferromagnetism consists of
discrete variables called spins arranged on the vertices of the
graph $G$. Each spin can take values ${\pm 1}$ and only interacts
with its nearest neighbors. Configuration of spins at two adjacent
vertices $i$ and $j$ has energy $J>0$ if the spins have opposite
values, and $-J$ if the values are the same. Let $|V(G)|=N$, and
let $\vec\sigma = (\sigma_1,...,\sigma_N)$  denote the
configuration of spins, with $\sigma_i\in\{\pm 1\}$. The total
energy of the system in configuration $\vec\sigma$ is then
$$
E(\vec\sigma) =-J\sum_{i\sim j}\sigma_i \sigma_j ,
$$
where $i \sim j$ means that the vertices $i$ and $j$ are adjacent
in $G$. The probability of a particular configuration at
temperature $T$ is given by
$$
\mathbb{P}(\vec\sigma) = \frac{1}{Z} \exp (-\beta E(\vec\sigma)),
$$
where $\beta$ is the \lq\lq inverse temperature\rq\rq
conventionally defined as $\beta\equiv1/(k_B T)$, and $k_B$
denotes the Boltzmann constant. As usual in statistical physics,
the normalizing constant $Z$, that makes the distribution above a
probability measure, is called the \emph{partition function}:
$$
Z=\sum_{\vec\sigma}\exp(-\beta E(\vec\sigma)).
$$
It is known \cite{potts} that the partition function $Z$ of the
Ising model on $G$ can be obtained by evaluating the Tutte
polynomial $T(G;x,y)$ on the hyperbola $(x-1)(y-1)=2$. More
precisely, one has:
$$
Z = 2(e^{2\beta J}-1)^{|V(G)|-1}e^{-\beta
J|E(G)|}T\left(G;\frac{e^{2\beta J}+1}{e^{2\beta J}-1},e^{2\beta
J}\right).
$$
In the next sections we will explicitly verify this correspondence
for the Schreier graphs of both the Grigorchuk and Basilica
groups, using the computations of the partition functions made in
\cite{noiising}.

\subsection{Groups of automorphisms of rooted regular trees}

We recall some basic facts about self-similar groups. Let $T_q$ be
the infinite regular rooted tree of degree $q$, i.e., the rooted
tree in which each vertex has $q$ children. Each vertex of the
$n$-th level of the tree can be regarded as a word of length $n$
in the alphabet $X=\{0,1,\ldots, q-1\}$. Moreover, one can
identify the set $X^{\omega}$ of infinite words in $X$ with the
set $\partial T_q$ of infinite geodesic rays starting at the root
of $T_q$. Next, let $S<Aut(T_q)$ be a group acting on $T_q$ by
automorphisms generated by a finite symmetric set of generators
$Y$. Moreover, suppose that the action is transitive on each level
of the tree.

\begin{defi}\label{defischreiernovembre}
The $n$-th {\it Schreier graph} $\Sigma_n$ of the action of $S$ on
$T_q$, with respect to the generating set $Y$, is a graph whose
vertex set coincides with the set of vertices of the $n$-th level
of the tree, and two vertices $u,v$ are adjacent if and only if
there exists $s\in Y$ such that $s(u)=v$. If this is the case, the
edge joining $u$ and $v$ is labelled by $s$.
\end{defi}
The vertices of $\Sigma_n$ are labelled by words of length $n$ in
$X$ and the edges are labelled by elements of $Y$. The Schreier
graph is thus a regular graph of degree $|Y|$ with $q^n$ vertices,
and it is connected, since the action of $S$ is level-transitive.

\begin{defi}\cite{volo}\label{defiselfsimilar}
A finitely generated group $S<Aut(T_q)$ is {\it self-similar} if,
for all $g\in S, x\in X$, there exist $h\in S, y\in X$ such that
$$
g(xw)=yh(w),
$$
for all finite words $w$ in the alphabet $X$.
\end{defi}

Self-similarity implies that $S$ can be embedded into the wreath
product $Sym(q)\wr S$, where $Sym(q)$ denotes the symmetric group
on $q$ elements, so that any automorphism $g\in S$ can be
represented as
$$
g=\tau(g_0,\ldots,g_{q-1}),
$$
where $\tau\in Sym(q)$ describes the action of $g$ on the first
level of $T_q$ and $g_i\in S, i=0,...,q-1$, is the restriction of
$g$ on the full subtree of $T_q$ rooted at the vertex $i$ of the
first level of $T_q$ (observe that any such subtree is isomorphic
to $T_q$). Hence, if $x\in X$ and $w$ is a finite word in $X$, we
have
$$
g(xw)=\tau(x)g_x(w).
$$
In the next sections, the Schreier graphs of the Grigorchuk group and of the Basilica group will be described. For both of them, we recall some substitutional rules that allow to recursively construct these sequences of graphs, starting from the Schreier graph associated with the action of the group on the first level of the rooted binary tree.

\section{The Tutte polynomial of the Schreier graphs of the Grigorchuk group}\label{sezione Grigor}

The Grigorchuk group admits the following description as a
self-similar group of automorphisms of the rooted binary tree. It is
generated by the elements
$$
a=\epsilon(id,id), \qquad b=e(a,c), \qquad c=e(a,d), \qquad
d=e(id,b),
$$
where $e$ and $\epsilon$ are respectively the trivial and the
non-trivial permutations in $Sym(2)$. Note that each generator is
an involution.

The following substitutional rules describe how to construct the
graph $\Gamma_{n+1}$ from $\Gamma_n$ \cite{hecke, grigorchuk}.
More precisely, the construction consists in replacing the
labelled subgraphs of $\Gamma_{n}$ on the top of the picture by
new labelled graphs (on the bottom).
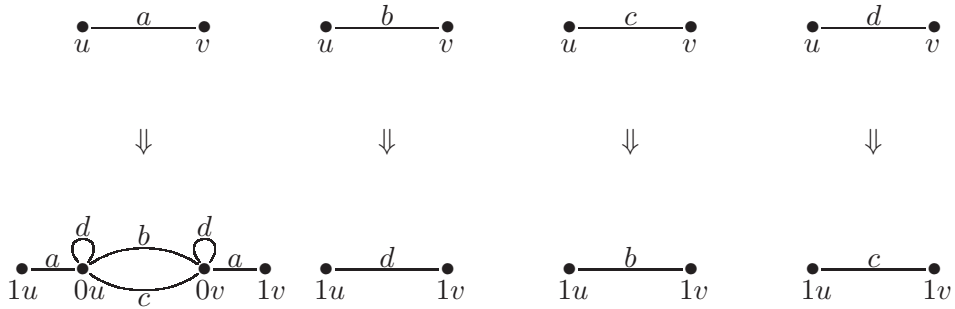
\begin{center}
\begin{picture}(400,110)
\letvertex A=(65,100)\letvertex B=(105,100)\letvertex C=(145,100)
\letvertex D=(185,100)\letvertex E=(225,100)\letvertex F=(265,100)\letvertex G=(305,100)
\letvertex H=(345,100)\letvertex I=(45,20)\letvertex L=(65,20)\letvertex M=(105,20)\letvertex N=(125,20)
\letvertex c=(145,20)
\letvertex d=(185,20)\letvertex e=(225,20)\letvertex f=(265,20)\letvertex g=(305,20)
\letvertex h=(345,20)

\put(82,60){$\Downarrow$}\put(162,60){$\Downarrow$}\put(242,60){$\Downarrow$}\put(322,60){$\Downarrow$}

\put(62,92){$u$} \put(102,92){$v$}\put(142,92){$u$}
\put(182,92){$v$}\put(222,92){$u$} \put(262,92){$v$}
\put(302,92){$u$}\put(342,92){$v$}

\put(40,10){$1u$} \put(62,10){$0u$}\put(102,10){$0v$}
\put(122,10){$1v$}

\put(141,10){$1u$} \put(181,10){$1v$}\put(221,10){$1u$}
\put(261,10){$1v$} \put(301,10){$1u$}\put(341,10){$1v$}

\drawvertex(A){$\bullet$}\drawvertex(B){$\bullet$}
\drawvertex(C){$\bullet$}\drawvertex(D){$\bullet$}
\drawvertex(E){$\bullet$}\drawvertex(F){$\bullet$}
\drawvertex(G){$\bullet$}\drawvertex(H){$\bullet$}
\drawvertex(I){$\bullet$}\drawvertex(L){$\bullet$}
\drawvertex(M){$\bullet$}\drawvertex(N){$\bullet$}
\drawvertex(c){$\bullet$}\drawvertex(d){$\bullet$}
\drawvertex(e){$\bullet$}\drawvertex(f){$\bullet$}
\drawvertex(g){$\bullet$}\drawvertex(h){$\bullet$}

\drawundirectedloop(L){$d$}\drawundirectedloop(M){$d$}

\drawundirectededge(A,B){$a$}\drawundirectededge(C,D){$b$}
\drawundirectededge(E,F){$c$} \drawundirectededge(G,H){$d$}
\drawundirectededge(c,d){$d$}\drawundirectededge(e,f){$b$}
\drawundirectededge(g,h){$c$}
\drawundirectededge[r](L,I){$a$}\drawundirectededge(M,N){$a$}
\drawundirectedcurvededge(L,M){$b$}\drawundirectedcurvededge[b](L,M){$c$}
\end{picture}
\end{center}
The starting point is the Schreier graph $\Gamma_1$ of the first
level.
\begin{center}
\begin{picture}(200,40)
\letvertex A=(70,25)\letvertex B=(130,25)

\put(67,16){$0$}\put(127,16){$1$}\put(15,22){$\Gamma_1$}

\drawvertex(A){$\bullet$}\drawvertex(B){$\bullet$}

\drawundirectedloop[l](A){$b,c,d$}\drawundirectedloop[r](B){$b,c,d$}
\drawundirectededge(A,B){$a$}
\end{picture}
\end{center}
In computing the Tutte polynomial of $\Gamma_n$, we are interested
in the unlabelled graph. We draw here the graphs $\Gamma_n$, for
$n=1,2,3$.
\begin{center}
\begin{picture}(300,40)
\letvertex A=(60,20)\letvertex B=(90,20)
\letvertex C=(155,20)\letvertex D=(185,20)
\letvertex E=(225,20)\letvertex F=(255,20)
\drawundirectedloop[l](A){}\drawundirectedloop[r](B){}\drawundirectedloop[l](C){}\drawundirectedloop(D){}\drawundirectedloop(E){}
\drawundirectedloop[r](F){}\drawundirectedloop[t](A){}\drawundirectedloop[b](A){}\drawundirectedloop[t](B){}\drawundirectedloop[b](B){}
\drawundirectedloop[t](C){}\drawundirectedloop[b](C){}\drawundirectedloop[t](F){}\drawundirectedloop[b](F){}

\drawundirectededge(A,B){} \drawundirectededge(C,D){}
\drawundirectededge(E,F){}
\drawundirectedcurvededge(D,E){}\drawundirectedcurvededge(E,D){}

\put(30,18){$\Gamma_1$} \put(273,18){$\Gamma_2$}

\drawvertex(A){$\bullet$}\drawvertex(B){$\bullet$}
\drawvertex(C){$\bullet$}\drawvertex(D){$\bullet$}
\drawvertex(E){$\bullet$}\drawvertex(F){$\bullet$}
\end{picture}
\end{center}

\begin{center}
\begin{picture}(300,40)
\letvertex A=(30,20)\letvertex B=(60,20)\letvertex C=(100,20)\letvertex D=(130,20)
\letvertex E=(170,20)\letvertex F=(200,20)\letvertex G=(240,20)\letvertex H=(270,20)

\drawundirectededge(A,B){} \drawundirectededge(C,D){}
\drawundirectededge(E,F){} \drawundirectededge(G,H){}

\put(290,18){$\Gamma_3$}

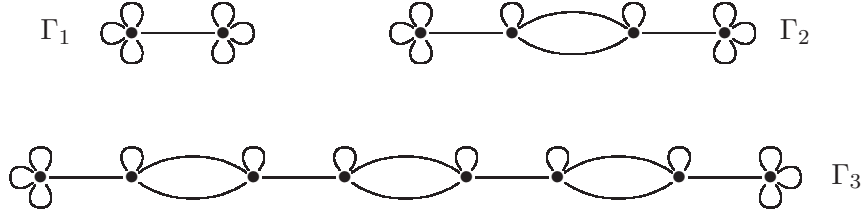
\drawvertex(A){$\bullet$}\drawvertex(B){$\bullet$}
\drawvertex(C){$\bullet$}\drawvertex(D){$\bullet$}
\drawvertex(E){$\bullet$}\drawvertex(F){$\bullet$}
\drawvertex(G){$\bullet$}\drawvertex(H){$\bullet$}

\drawundirectedcurvededge(B,C){} \drawundirectedcurvededge(C,B){}
\drawundirectedcurvededge(D,E){}\drawundirectedcurvededge(E,D){}
\drawundirectedcurvededge(F,G){} \drawundirectedcurvededge(G,F){}
\drawundirectedloop[l](A){}\drawundirectedloop(B){}\drawundirectedloop(C){}\drawundirectedloop(D){}\drawundirectedloop(E){}
\drawundirectedloop(F){}\drawundirectedloop(G){}\drawundirectedloop[r](H){}\drawundirectedloop[t](A){}\drawundirectedloop[b](A){}
\drawundirectedloop[t](H){}\drawundirectedloop[b](H){}
\end{picture}
\end{center}
In general, one can check by using the substitutional rules that
$\Gamma_n$ has a linear shape, obtained by alternating bridges and
$2$-cycles. More precisely, it is easy to prove the following
equalities:
$$
|V(\Gamma_n)| = 2^n \qquad \qquad |E(\Gamma_n)|=5\cdot2^{n-1}+2.
$$
As regards edges, note that there are $2^{n-1}$ bridges,
$2^{n-1}-1$ cycles of length 2 and $2^{n}+4$ loops, since there is
a loop rooted at each vertex, except for the outmost vertices,
where three loops are rooted.

Since many computations are trivial for graphs with loops, it is
convenient to consider $\Gamma_n^\ast$, defined as the graph obtained from $\Gamma_n$ by erasing loops. Thus, in this case, we have
$$
|V(\Gamma_n^\ast)| = 2^n \qquad \qquad |E(\Gamma_n^
\ast)|=3\cdot2^{n-1}-2.
$$
For every $n \geq 1 $, denote by $T_n(x,y)$ the Tutte polynomial
$T(\Gamma_n;x,y)$ of $\Gamma_n$ and  by $T_n^\ast(x,y)$ the Tutte
polynomial $T(\Gamma_n^\ast;x,y)$ of $\Gamma_n^\ast$.

\begin{teo}\label{thmtuttegrigorchuk}
For each $n\geq 1$, the Tutte polynomial of the graph $\Gamma_n$
is
$$
T_n(x,y) = y^{2^n+4} x^{2^{n-1}} (y+x)^{2^{n-1}-1}.
$$
\end{teo}

\begin{proof}
It suffices to apply the multiplicative property, keeping in mind
that each loop, bridge or $2$-cycle contributes by a factor $y, x$
or $(y+x)$, respectively.
\end{proof}

\begin{cor}\label{cor.tutte.grigSTAR}
For each $n\geq 1$, the Tutte polynomial of the graph $\Gamma_n^\ast$  is
$$
T_n^{\ast}(x,y) = x^{2^{n-1}}  (y+x)^{2^{n-1}-1}.
$$
\end{cor}

Let us start by writing the reliability polynomial
$R(\Gamma_n,p)$.
\begin{prop}\label{prop.reliabGRIG}
For each $n\geq 1$, the reliability polynomial $R(\Gamma_n,p)$ is
given by
$$
R(\Gamma_n,p)
 = p^{2^n-1}(2-p)^{2^{n-1}-1}.
$$
\end{prop}

\begin{proof}
It suffices to apply Equation (1) of Theorem \ref{twopolynomials}.
\end{proof}

\begin{os}\rm
Note that the existence of loops does not change the reliability polynomial; therefore  $R(\Gamma_n,p)=R(\Gamma_n^\ast,p)$, as one can directly check.
\end{os}

As regards the complexity of $\Gamma_n$, the following proposition holds.

\begin{prop}\label{pro.complexi.GRIG}
The complexity of $\Gamma_n$ is $2^{2^{n-1}-1}$.
\end{prop}

\begin{proof}
According with Formula (1) of   Theorem \ref{evaluations}, it suffices to compute $T_n(1,1)$.
\end{proof}

\begin{os}\rm
The value of $\tau(\Gamma_n)$ has the following interpretation:
each bridge of $\Gamma_n$ must belong to any spanning subtree of
$\Gamma_n$. On the other hand, a spanning subtree of $\Gamma_n$
must contain exactly one edge of each $2$-cycle. Then the result
follows, since the number of $2$-cycles of $\Gamma_n$ is
$2^{n-1}-1$ and we have two choices for each $2$-cycle. Also observe
that loops do not contribute to $\tau(\Gamma_n)$; therefore,
$T_n(1,1)= T_n^\ast(1,1)$ and so $\tau(\Gamma_n) =
\tau(\Gamma_n^\ast)$, for each $n\geq 1$.
\end{os}

\begin{cor}
The asymptotic growth constant of the spanning trees of $\Gamma_n$
is $\frac{1}{2}\log 2$.
\end{cor}

\begin{proof}
It suffices to compute
$$
\lim_{n\to \infty}\frac{\log(\tau(\Gamma_n))}{|V(\Gamma_n)|},
$$
with $|V(\Gamma_n)|=2^n$.
\end{proof}

Evaluating $T_n(x,y)$ in $(1,2)$ provides the number of connected
spanning subgraphs of $\Gamma_n$. The following
proposition holds.

\begin{prop}\label{prop.conn.span.subgra.GRIG}
The number of connected spanning subgraphs of $\Gamma_n$ is
$2^{2^n+4}\cdot 3^{2^{n-1}-1}$.
\end{prop}

\begin{proof}
It suffices to apply Formula (2) of Theorem \ref{evaluations}.
\end{proof}

\begin{os}\rm
The value that we have found for the number of connected spanning
subgraphs of $\Gamma_n$ has the following interpretation: a
connected spanning subgraph of $\Gamma_n$ necessarily contains
each bridge of the graph. On the other hand, both the edges or
only one edge of each $2$-cycle must belong to the subgraph (if no
edge of a cycle belongs to the subgraph, then this subgraph is not
connected), so that, for each of the  $2^{n-1}-1$ cycles of length 2,
we have three possibilities. Finally, a connected spanning subgraph can also contain loops and so we have two possibilities for each of the  $2^n+4$ loops.
\end{os}

Another interesting computation concerns the number of spanning
forests of $\Gamma_n$, which is given by $T_n(2,1)$.
\begin{prop}\label{prop.conn.span.forest.GRIG}
The number of spanning forests of $\Gamma_n$ is $2^{2^{n-1}}\cdot
3^{2^{n-1}-1}$.
\end{prop}

\begin{proof}
It suffices to apply Formula (3) of Theorem \ref{evaluations}.
\end{proof}

\begin{os}\rm
The value that we have found for the number of spanning forests of
$\Gamma_n$ has the following interpretation: a spanning forest of
$\Gamma_n$ cannot contain loops nor both the edges of a $2$-cycle,
since this would produce a cycle. Therefore, no edges or only one
edge of each $2$-cycle must belong to the forest. On the other
hand, each bridge can belong to a spanning forest of $\Gamma_n$.
Since the number of $2$-cycles is $2^{n-1}-1$ and the number of
bridges is $2^{n-1}$, we get the result.
\end{os}

Next, we explicitly verify that by evaluating the Tutte polynomial
of $\Gamma_n$ in $(2,2)$ one gets $2^{|E(\Gamma_n)|}$ (see Formula
(4) of Theorem \ref{evaluations}).

\begin{prop}
For each $n\geq 1$, one has $T_n(2,2)=
2^{|E(\Gamma_{n})|}=2^{5\cdot 2^{n-1}+2}$.
\end{prop}

\begin{proof}
By definition of $T_n(x,y)$, one has:
\begin{eqnarray*}
T_n(2,2) =2^{2^n+4}\cdot  2^{2^{n-1}} \cdot 4^{2^{n-1}-1} = 2^{5\cdot
2^{n-1}+2}.
\end{eqnarray*}
\end{proof}

Finally, by evaluating the Tutte polynomial of $\Gamma_n$ in
$(2,0)$, we investigate the number of acyclic orientations of
$\Gamma_n$. Observe that, whenever we have loops, the number of
possible acyclic orientations on the graphs is $0$. Therefore, we
consider the graphs $\{\Gamma_n^\ast\}_{n \geq 1 }$ without loops, whose Tutte polynomial is $T_n^{\ast}(x,y) = x^{2^{n-1}}
(x+y)^{2^{n-1}-1}$.

\begin{prop}\label{prop.acyclic.orien.GRIG}
The number of acyclic orientations on $\Gamma_n^{\ast}$ is
$2^{2^n-1}$.
\end{prop}

\begin{proof}
By definition of $T_n^{\ast}(x,y)$, one has:
\begin{eqnarray*}
T_n^{\ast}(2,0) = 2^{2^{n-1}}\cdot2^{2^{n-1}-1}=2^{2^{n}-1}.
\end{eqnarray*}
\end{proof}

\begin{os}\rm
The value that we have found for the number of acyclic
orientations of $\Gamma_n^{\ast}$ has the following
interpretation: we have two possible orientations for each bridge,
giving the factor $2^{2^{n-1}}$. Then, each $2$-cycle can
receive four orientations, as shown in the following picture.
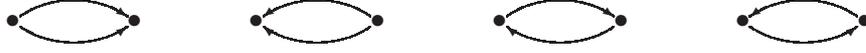
\begin{center}
\begin{picture}(300,8)
\letvertex A=(10,4)\letvertex B=(50,4)
\drawvertex(A){$\bullet$}\drawvertex(B){$\bullet$}
\drawcurvededge(A,B){} \drawcurvededge[b](A,B){}

\letvertex AA=(90,4)\letvertex BB=(130,4)
\drawvertex(AA){$\bullet$}\drawvertex(BB){$\bullet$}
\drawcurvededge(BB,AA){} \drawcurvededge[b](BB,AA){}

\letvertex a=(170,4)\letvertex b=(210,4)
\drawvertex(a){$\bullet$}\drawvertex(b){$\bullet$}
\drawcurvededge(a,b){} \drawcurvededge(b,a){}

\letvertex aA=(250,4)\letvertex bB=(290,4)
\drawvertex(aA){$\bullet$}\drawvertex(bB){$\bullet$}
\drawcurvededge[b](aA,bB){} \drawcurvededge[b](bB,aA){}
\end{picture}
\end{center}
Only the first two orientations are acyclic, and so the $2$-cycles
give a contribution equal to $2^{2^{n-1}-1}$.
\end{os}

Since $\Gamma_n$ has loops, it does not admit any proper coloring,
so that we investigate the chromatic polynomial of
$\Gamma_n^{\ast}$.

\begin{prop}\label{prop.chromatic poly.GRIG}
For each $n\geq 1$, the chromatic polynomial $\chi_n(\lambda)$ of
$\Gamma_n^{\ast}$ is
$$
\chi_n(\lambda) = -\lambda(1-\lambda)^{2^n-1}.
$$
\end{prop}

\begin{proof}
By applying Equation (2) of Theorem \ref{twopolynomials}, one gets
\begin{eqnarray*}
\chi_n(\lambda)=
(-1)^{2^n-1}\lambda(1-\lambda)^{2^{n-1}}(1-\lambda)^{2^{n-1}-1}=-\lambda(1-\lambda)^{2^n-1}.
\end{eqnarray*}
\end{proof}

\begin{os}\rm
Note that $\chi_n(2)=2$, for each $n\geq 1$, according to the fact
that the graph is bipartite and so uniquely $2$-colorable.
\end{os}

We end this section by investigating the relationship between the
evaluation of the Tutte polynomial of the Schreier graph
$\Gamma_n^{\ast}$ on the hyperbola $(x-1)(y-1)=2$ and the
partition function of the Ising model on the same graph. In
\cite[Theorem 2.1]{noiising}, the partition function of the Ising
model on $\Gamma_n^{\ast}$ has been described as
$$
Z_n=2^{2^n} \cosh(\beta J)^{3\cdot 2^{n-1}-2}
\left(1+\tanh^2\left(\beta J\right)\right)^{2^{n-1}-1}.
$$

\begin{teo}\label{thmisinggrig}
For each $n\geq 1$, one has
\begin{eqnarray}\label{isingbase}
2(e^{2\beta J}-1)^{|V(\Gamma_n^{\ast})|-1}\cdot e^{-\beta
J|E(\Gamma_n^{\ast})|}\cdot T_n^{\ast}\left(\frac{e^{2\beta
J}+1}{e^{2\beta J}-1},e^{2\beta J}\right) = Z_n.
\end{eqnarray}
\end{teo}

\begin{proof}
Recall that $|E(\Gamma_n^{\ast})| = 3\cdot 2^{n-1}-2$ and
$|V(\Gamma_n^{\ast})|= 2^n$. Let $e^{\beta J}=t$, so that Equation
\eqref{isingbase} can be written as
\begin{eqnarray}\label{equalityising}
\frac{2(t^2-1)^{2^n-1}}{t^{3\cdot
2^{n-1}-2}}\cdot T_n^{\ast}\left(\frac{t^2+1}{t^2-1},t^2\right)=
2^{2^n} \left(\frac{t^2+1}{2t}\right)^{3\cdot 2^{n-1}-2 }
\left(1+ \left(\frac{t^2-1}{t^2+1}\right)^2\right)^{2^{n-1}-1}.
\end{eqnarray}
One can directly check that
$$
T_n^{\ast}\left(\frac{t^2+1}{t^2-1},t^2\right)=
\frac{(t^2+1)^{2^{n-1}}(t^4+1)^{2^{n-1}-1}}{(t^2-1)^{2^n-1}}.
$$
Then, it is not difficult to prove that both sides of Equation
\eqref{equalityising} are equal to
$$
\frac{2(t^2+1)^{2^{n-1}}(t^4+1)^{2^{n-1}-1}}{t^{3\cdot
2^{n-1}-2}}.
$$
\end{proof}

\section{The Tutte polynomial of the Schreier graphs of the Basilica group}\label{sezio Basilica}

The Basilica group is a self-similar group of automorphisms of the
rooted binary tree generated by the elements
$$
a=e(b,id), \qquad b=\epsilon(a,id).
$$
The associated Schreier graphs can be recursively constructed via
the following substitutional rules:
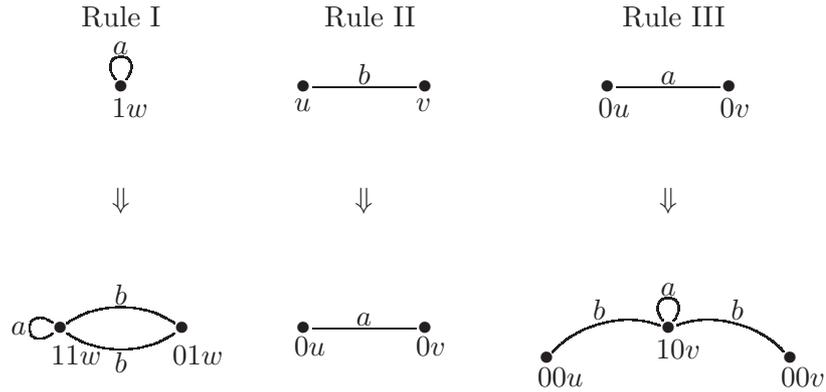
\begin{center}
\begin{picture}(400,128)

\put(97,120){Rule I}\put(177,120){Rule II}\put(275,120){Rule III}

\letvertex A=(110,100)\letvertex B=(90,20)\letvertex C=(130,20)

\letvertex D=(170,100)\letvertex E=(210,100)\letvertex F=(170,20)\letvertex G=(210,20)
\letvertex H=(270,100)\letvertex I=(310,100)\letvertex L=(250,10)\letvertex M=(290,20)\letvertex N=(330,10)

\put(107,60){$\Downarrow$}\put(187,60){$\Downarrow$}\put(287,60){$\Downarrow$}

\put(107,90){$1w$}\put(87,8){$11w$}\put(127,8){$01w$}

\put(167,92){$u$}\put(207,92){$v$}\put(167,11){$0u$}\put(207,11){$0v$}

\put(267,90){$0u$}\put(307,90){$0v$}
\put(247,1){$00u$}\put(286,10){$10v$}\put(327,1){$00v$}


\drawvertex(A){$\bullet$}\drawvertex(B){$\bullet$}
\drawvertex(C){$\bullet$}\drawvertex(D){$\bullet$}
\drawvertex(E){$\bullet$}\drawvertex(F){$\bullet$}
\drawvertex(G){$\bullet$}\drawvertex(H){$\bullet$}
\drawvertex(I){$\bullet$}\drawvertex(L){$\bullet$}
\drawvertex(M){$\bullet$}\drawvertex(N){$\bullet$}

\drawundirectedloop(A){$a$}\drawundirectedloop[l](B){$a$}
\drawundirectedcurvededge(B,C){$b$}\drawundirectedcurvededge(C,B){$b$}

\drawundirectededge(D,E){$b$} \drawundirectededge(F,G){$a$}
\drawundirectededge(H,I){$a$}
\drawundirectedcurvededge(L,M){$b$}\drawundirectedloop(M){$a$}
\drawundirectedcurvededge(M,N){$b$}
\end{picture}
\end{center}
The starting point is the Schreier graph $\B_1$ of the first
level:
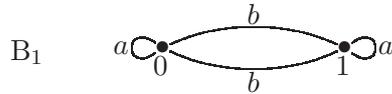
\begin{center}
\begin{picture}(200,20)
\letvertex A=(70,12)\letvertex B=(130,12)

\put(67,3){$0$}\put(127,3){$1$}\put(20,8){$\B_1$}

\drawvertex(A){$\bullet$}\drawvertex(B){$\bullet$}

\drawundirectedloop[l](A){$a$}\drawundirectedloop[r](B){$a$}
\drawundirectedcurvededge(A,B){$b$}\drawundirectedcurvededge(B,A){$b$}
\end{picture}
\end{center}

As in the case of the Grigorchuk group, we are interested in the
unlabelled Schreier graphs. The following pictures of graphs $\B_n$
for $n=1,2,3,4,5,6$ give an idea of how Schreier graphs of
the Basilica group look like. See \cite{schreierbasilica} for a
comprehensive analysis of finite and infinite Schreier graphs of
this group. Note also that $\{\B_n\}_{n\geq 1}$ is an
approximating sequence for the Julia set of the polynomial
$z^2-1$, the famous \lq\lq Basilica\rq\rq fractal (see
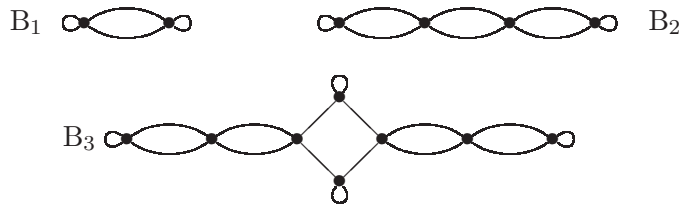
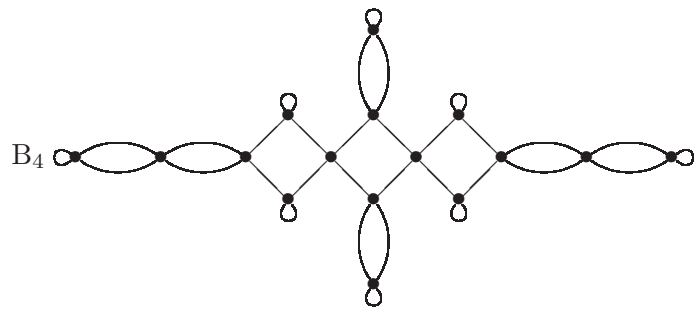
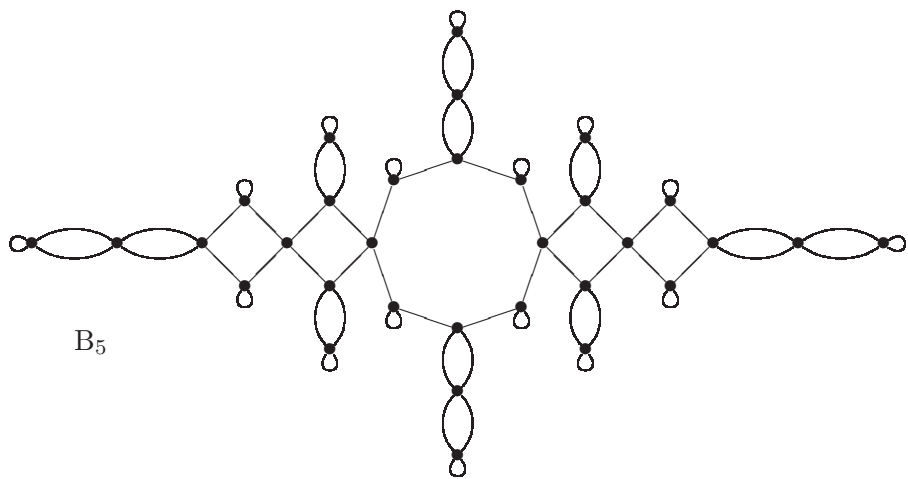
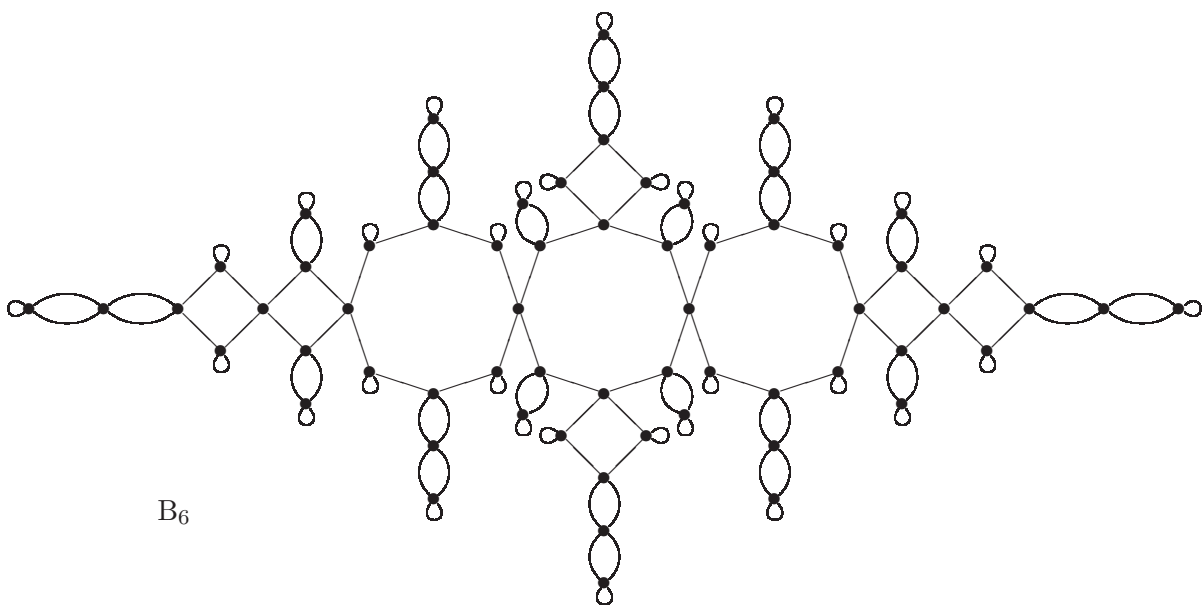
\cite{volo}). \unitlength=0.28mm
\begin{center}
\begin{picture}(300,25)
\letvertex A=(30,15)\letvertex B=(70,15)\letvertex C=(150,15)\letvertex D=(190,15)
\letvertex E=(230,15)\letvertex F=(270,15)

\drawvertex(A){$\bullet$}\drawvertex(B){$\bullet$}
\drawvertex(C){$\bullet$}\drawvertex(D){$\bullet$}
\drawvertex(E){$\bullet$}\drawvertex(F){$\bullet$}

\drawundirectedcurvededge(A,B){} \drawundirectedcurvededge(B,A){}
\drawundirectedcurvededge(C,D){} \drawundirectedcurvededge(D,C){}
\drawundirectedcurvededge(D,E){} \drawundirectedcurvededge(E,D){}
\drawundirectedcurvededge(E,F){} \drawundirectedcurvededge(F,E){}
\put(-5,12){$\B_1$} \put(295,12){$\B_2$}

\drawundirectedloop[l](A){}\drawundirectedloop[r](B){}\drawundirectedloop[l](C){}
\drawundirectedloop[r](F){}
\end{picture}
\end{center}

\begin{center}
\begin{picture}(300,57)
\letvertex A=(50,30)\letvertex B=(90,30)\letvertex C=(130,30)\letvertex D=(150,50)
\letvertex E=(150,10)\letvertex F=(170,30)\letvertex G=(210,30)\letvertex H=(250,30)

\drawvertex(A){$\bullet$}\drawvertex(B){$\bullet$}
\drawvertex(C){$\bullet$}\drawvertex(D){$\bullet$}
\drawvertex(E){$\bullet$}\drawvertex(F){$\bullet$}
\drawvertex(G){$\bullet$}\drawvertex(H){$\bullet$}

\drawundirectededge(C,D){} \drawundirectededge(E,C){}
\drawundirectededge(F,E){} \drawundirectededge(D,F){}

\drawundirectedcurvededge(A,B){} \drawundirectedcurvededge(B,A){}
\drawundirectedcurvededge(B,C){} \drawundirectedcurvededge(C,B){}
\drawundirectedcurvededge(F,G){} \drawundirectedcurvededge(G,F){}
\drawundirectedcurvededge(G,H){} \drawundirectedcurvededge(H,G){}
\put(20,27){$\B_3$}
\drawundirectedloop[l](A){}\drawundirectedloop(D){}\drawundirectedloop[b](E){}
\drawundirectedloop[r](H){}
\end{picture}
\end{center}
\vspace{0,4cm}
\begin{center}
\begin{picture}(300,140)
\put(-20,67){$\B_4$}

\letvertex A=(10,70)
\letvertex B=(50,70)\letvertex C=(90,70)\letvertex D=(110,90)\letvertex E=(110,50)
\letvertex F=(130,70)\letvertex G=(150,90)\letvertex H=(150,50)\letvertex I=(170,70)
\letvertex J=(150,130)\letvertex K=(150,10)
\letvertex L=(190,90)
\letvertex M=(190,50)\letvertex N=(210,70)
\letvertex O=(250,70)\letvertex P=(290,70)

\drawvertex(A){$\bullet$}\drawvertex(B){$\bullet$}
\drawvertex(C){$\bullet$}\drawvertex(D){$\bullet$}
\drawvertex(E){$\bullet$}\drawvertex(F){$\bullet$}
\drawvertex(G){$\bullet$}\drawvertex(H){$\bullet$}
\drawvertex(I){$\bullet$}\drawvertex(L){$\bullet$}
\drawvertex(M){$\bullet$}\drawvertex(N){$\bullet$}
\drawvertex(O){$\bullet$}\drawvertex(P){$\bullet$}
\drawvertex(J){$\bullet$}\drawvertex(K){$\bullet$}

\drawundirectedcurvededge(A,B){}\drawundirectedcurvededge(B,A){}
\drawundirectedcurvededge(B,C){}\drawundirectedcurvededge(C,B){}
\drawundirectededge(C,D){} \drawundirectededge(D,F){}
\drawundirectededge(F,E){} \drawundirectededge(E,C){}
\drawundirectededge(F,G){} \drawundirectededge(G,I){}
\drawundirectededge(I,H){} \drawundirectededge(H,F){}
\drawundirectededge(I,L){} \drawundirectededge(L,N){}
\drawundirectededge(N,M){} \drawundirectededge(M,I){}
\drawundirectedcurvededge(G,J){}\drawundirectedcurvededge(J,G){}
\drawundirectedcurvededge(H,K){}\drawundirectedcurvededge(K,H){}
\drawundirectedcurvededge(N,O){}\drawundirectedcurvededge(O,N){}
\drawundirectedcurvededge(O,P){}\drawundirectedcurvededge(P,O){}
\drawundirectedloop[l](A){}\drawundirectedloop(D){}\drawundirectedloop[b](E){}
\drawundirectedloop(J){}\drawundirectedloop[b](K){}\drawundirectedloop(L){}\drawundirectedloop[b](M){}
\drawundirectedloop[r](P){}
\end{picture}
\end{center}
\vspace{0,3cm}
\begin{center}
\begin{picture}(400,210)
\put(20,60){$\B_5$}
\letvertex A=(0,110)\letvertex B=(40,110)\letvertex C=(80,110)\letvertex D=(100,130)
\letvertex E=(100,90)\letvertex F=(120,110)\letvertex G=(140,130)\letvertex H=(140,160)
\letvertex I=(160,110)\letvertex L=(140,90)\letvertex M=(140,60)\letvertex N=(170,140)
\letvertex O=(200,150)\letvertex R=(230,140)\letvertex S=(240,110)\letvertex T=(230,80)
\letvertex U=(200,70)\letvertex V=(170,80)\letvertex P=(200,180)\letvertex Q=(200,210)
\letvertex Z=(200,40)\letvertex J=(200,10)\letvertex K=(260,130)\letvertex X=(280,110)
\letvertex W=(260,90)\letvertex g=(260,160)\letvertex h=(260,60)\letvertex c=(300,130)
\letvertex Y=(300,90)\letvertex d=(320,110)\letvertex e=(360,110)\letvertex f=(400,110)

\drawvertex(A){$\bullet$}\drawvertex(B){$\bullet$}
\drawvertex(C){$\bullet$}\drawvertex(D){$\bullet$}
\drawvertex(E){$\bullet$}\drawvertex(F){$\bullet$}
\drawvertex(G){$\bullet$}\drawvertex(H){$\bullet$}
\drawvertex(I){$\bullet$}\drawvertex(L){$\bullet$}
\drawvertex(M){$\bullet$}\drawvertex(N){$\bullet$}
\drawvertex(O){$\bullet$}\drawvertex(P){$\bullet$}
\drawvertex(J){$\bullet$}\drawvertex(K){$\bullet$}
\drawvertex(Q){$\bullet$}\drawvertex(R){$\bullet$}
\drawvertex(S){$\bullet$}\drawvertex(T){$\bullet$}
\drawvertex(U){$\bullet$}\drawvertex(V){$\bullet$}
\drawvertex(W){$\bullet$}\drawvertex(X){$\bullet$}
\drawvertex(Y){$\bullet$}\drawvertex(Z){$\bullet$}
\drawvertex(g){$\bullet$}\drawvertex(h){$\bullet$}
\drawvertex(c){$\bullet$}\drawvertex(f){$\bullet$}
\drawvertex(d){$\bullet$}\drawvertex(e){$\bullet$}

\drawundirectedcurvededge(A,B){}\drawundirectedcurvededge(B,A){}
\drawundirectedcurvededge(B,C){}\drawundirectedcurvededge(C,B){}
\drawundirectededge(C,D){} \drawundirectededge(D,F){}
\drawundirectededge(F,E){} \drawundirectededge(E,C){}

\drawundirectededge(F,G){} \drawundirectededge(G,I){}
\drawundirectededge(I,L){} \drawundirectededge(L,F){}

\drawundirectedcurvededge(G,H){}\drawundirectedcurvededge(H,G){}
\drawundirectedcurvededge(L,M){}\drawundirectedcurvededge(M,L){}

\drawundirectededge(I,N){} \drawundirectededge(N,O){}
\drawundirectededge(O,R){} \drawundirectededge(R,S){}
\drawundirectededge(S,T){} \drawundirectededge(T,U){}
\drawundirectededge(U,V){} \drawundirectededge(V,I){}

\drawundirectedcurvededge(O,P){}\drawundirectedcurvededge(P,O){}
\drawundirectedcurvededge(Q,P){}\drawundirectedcurvededge(P,Q){}

\drawundirectedcurvededge(U,Z){}\drawundirectedcurvededge(Z,U){}
\drawundirectedcurvededge(Z,J){}\drawundirectedcurvededge(J,Z){}

\drawundirectededge(S,K){} \drawundirectededge(K,X){}
\drawundirectededge(X,W){} \drawundirectededge(W,S){}
\drawundirectededge(X,c){} \drawundirectededge(c,d){}
\drawundirectededge(d,Y){} \drawundirectededge(Y,X){}

\drawundirectedcurvededge(d,e){}\drawundirectedcurvededge(e,d){}
\drawundirectedcurvededge(e,f){}\drawundirectedcurvededge(f,e){}
\drawundirectedcurvededge(K,g){}\drawundirectedcurvededge(g,K){}
\drawundirectedcurvededge(W,h){}\drawundirectedcurvededge(h,W){}

\drawundirectedloop[l](A){}\drawundirectedloop(D){}\drawundirectedloop[b](E){}
\drawundirectedloop(H){}\drawundirectedloop[b](M){}\drawundirectedloop(Q){}\drawundirectedloop[b](J){}
\drawundirectedloop(g){}\drawundirectedloop[b](h){}\drawundirectedloop(c){}\drawundirectedloop[b](Y){}
\drawundirectedloop[r](f){}\drawundirectedloop[b](V){}\drawundirectedloop(N){}\drawundirectedloop[b](T){}
\drawundirectedloop(R){}
\end{picture}
\end{center}

\begin{center}
\begin{turn}{90}
\begin{picture}(400,540)

\put(100,480){\begin{rotate}{270}$\B_6$
\end{rotate}}

\letvertex A=(200,0)\letvertex B=(200,35)\letvertex C=(200,70)\letvertex D=(180,90)
\letvertex E=(220,90)\letvertex F=(200,110)\letvertex G=(180,130)\letvertex H=(155,130)
\letvertex I=(220,130)\letvertex L=(245,130)\letvertex M=(200,150)\letvertex N=(170,160)
\letvertex O=(160,190)\letvertex P=(170,220)\letvertex Q=(200,230)\letvertex R=(230,220)
\letvertex S=(240,190)\letvertex T=(230,160)\letvertex U=(135,190)\letvertex V=(110,190)
\letvertex Z=(265,190)\letvertex J=(290,190)\letvertex K=(170,240)\letvertex X=(150,232)
\letvertex W=(120,270)\letvertex Y=(140,250)
\letvertex a=(95,270)\letvertex b=(70,270)\letvertex c=(140,290)\letvertex d=(160,270)
\letvertex e=(150,308)\letvertex f=(170,300)\letvertex g=(200,310)\letvertex h=(230,300)
\letvertex i=(250,308)\letvertex l=(240,270)\letvertex m=(230,240)\letvertex n=(250,232)
\letvertex o=(260,250)\letvertex p=(260,290)\letvertex q=(280,270)\letvertex r=(305,270)
\letvertex s=(330,270)\letvertex t=(170,320)\letvertex u=(160,350)\letvertex v=(170,380)
\letvertex z=(200,390)\letvertex j=(265,350)\letvertex k=(240,350)\letvertex x=(230,380)
\letvertex w=(290,350)\letvertex y=(230,320)
\letvertex AA=(135,350)\letvertex BB=(110,350)\letvertex CC=(155,410)\letvertex DD=(180,410)
\letvertex EE=(220,410)\letvertex FF=(245,410)\letvertex GG=(200,430)\letvertex HH=(180,450)
\letvertex II=(220,450)\letvertex LL=(200,470)\letvertex MM=(200,505)\letvertex NN=(200,540)

\drawvertex(A){$\bullet$}\drawvertex(B){$\bullet$}
\drawvertex(C){$\bullet$}\drawvertex(D){$\bullet$}
\drawvertex(E){$\bullet$}\drawvertex(F){$\bullet$}
\drawvertex(G){$\bullet$}\drawvertex(H){$\bullet$}
\drawvertex(I){$\bullet$}\drawvertex(L){$\bullet$}
\drawvertex(M){$\bullet$}\drawvertex(N){$\bullet$}
\drawvertex(O){$\bullet$}\drawvertex(P){$\bullet$}
\drawvertex(J){$\bullet$}\drawvertex(K){$\bullet$}
\drawvertex(Q){$\bullet$}\drawvertex(R){$\bullet$}
\drawvertex(S){$\bullet$}\drawvertex(T){$\bullet$}
\drawvertex(U){$\bullet$}\drawvertex(V){$\bullet$}
\drawvertex(W){$\bullet$}\drawvertex(X){$\bullet$}
\drawvertex(Y){$\bullet$}\drawvertex(Z){$\bullet$}

\drawvertex(a){$\bullet$}\drawvertex(b){$\bullet$}
\drawvertex(c){$\bullet$}\drawvertex(d){$\bullet$}
\drawvertex(e){$\bullet$}\drawvertex(f){$\bullet$}
\drawvertex(g){$\bullet$}\drawvertex(h){$\bullet$}
\drawvertex(i){$\bullet$}\drawvertex(l){$\bullet$}
\drawvertex(m){$\bullet$}\drawvertex(n){$\bullet$}
\drawvertex(o){$\bullet$}\drawvertex(p){$\bullet$}
\drawvertex(j){$\bullet$}\drawvertex(k){$\bullet$}
\drawvertex(q){$\bullet$}\drawvertex(r){$\bullet$}
\drawvertex(s){$\bullet$}\drawvertex(t){$\bullet$}
\drawvertex(u){$\bullet$}\drawvertex(v){$\bullet$}
\drawvertex(w){$\bullet$}\drawvertex(x){$\bullet$}
\drawvertex(y){$\bullet$}\drawvertex(z){$\bullet$}

\drawvertex(AA){$\bullet$}\drawvertex(BB){$\bullet$}
\drawvertex(CC){$\bullet$}\drawvertex(DD){$\bullet$}
\drawvertex(EE){$\bullet$}\drawvertex(FF){$\bullet$}
\drawvertex(GG){$\bullet$}\drawvertex(HH){$\bullet$}
\drawvertex(II){$\bullet$}\drawvertex(LL){$\bullet$}
\drawvertex(MM){$\bullet$}\drawvertex(NN){$\bullet$}

\drawundirectedcurvededge(A,B){}\drawundirectedcurvededge(B,A){}
\drawundirectedcurvededge(B,C){}\drawundirectedcurvededge(C,B){}

\drawundirectedcurvededge[r](H,G){}\drawundirectedcurvededge(G,H){}
\drawundirectedcurvededge[r](I,L){}\drawundirectedcurvededge(L,I){}
\drawundirectedcurvededge(V,U){}\drawundirectedcurvededge(U,V){}
\drawundirectedcurvededge(U,O){}\drawundirectedcurvededge(O,U){}
\drawundirectedcurvededge(S,Z){}\drawundirectedcurvededge(Z,S){}
\drawundirectedcurvededge(Z,J){}\drawundirectedcurvededge(J,Z){}
\drawundirectedcurvededge(a,b){}\drawundirectedcurvededge(b,a){}
\drawundirectedcurvededge(a,W){}\drawundirectedcurvededge(W,a){}

\drawundirectedcurvededge(q,r){}\drawundirectedcurvededge(r,q){}
\drawundirectedcurvededge(r,s){}\drawundirectedcurvededge(s,r){}
\drawundirectedcurvededge(AA,BB){}\drawundirectedcurvededge(BB,AA){}
\drawundirectedcurvededge(AA,u){}\drawundirectedcurvededge(u,AA){}
\drawundirectedcurvededge(k,j){}\drawundirectedcurvededge(j,k){}
\drawundirectedcurvededge(j,w){}\drawundirectedcurvededge(w,j){}
\drawundirectedcurvededge(CC,DD){}\drawundirectedcurvededge[r](DD,CC){}
\drawundirectedcurvededge(EE,FF){}\drawundirectedcurvededge[r](FF,EE){}

\drawundirectedcurvededge(LL,MM){}\drawundirectedcurvededge(MM,LL){}
\drawundirectedcurvededge(MM,NN){}\drawundirectedcurvededge(NN,MM){}
\drawundirectedcurvededge(X,K){}\drawundirectedcurvededge(K,X){}
\drawundirectedcurvededge(e,f){}\drawundirectedcurvededge(f,e){}
\drawundirectedcurvededge(h,i){}\drawundirectedcurvededge(i,h){}
\drawundirectedcurvededge(m,n){}\drawundirectedcurvededge(n,m){}

\drawundirectededge(C,D){} \drawundirectededge(D,F){}
\drawundirectededge(F,E){} \drawundirectededge(E,C){}

\drawundirectededge(F,G){} \drawundirectededge(G,M){}
\drawundirectededge(M,I){} \drawundirectededge(I,F){}

\drawundirectededge(M,N){} \drawundirectededge(N,O){}
\drawundirectededge(O,P){} \drawundirectededge(Q,P){}
\drawundirectededge(R,Q){} \drawundirectededge(R,S){}
\drawundirectededge(S,T){} \drawundirectededge(T,M){}

\drawundirectededge(K,Q){} \drawundirectededge(K,d){}
\drawundirectededge(d,f){} \drawundirectededge(g,f){}
\drawundirectededge(h,g){} \drawundirectededge(h,l){}
\drawundirectededge(l,m){} \drawundirectededge(Q,m){}

\drawundirectededge(d,Y){} \drawundirectededge(Y,W){}
\drawundirectededge(W,c){} \drawundirectededge(c,d){}

\drawundirectededge(l,p){} \drawundirectededge(p,q){}
\drawundirectededge(q,o){} \drawundirectededge(o,l){}

\drawundirectededge(t,g){} \drawundirectededge(t,u){}
\drawundirectededge(u,v){} \drawundirectededge(v,z){}
\drawundirectededge(z,x){} \drawundirectededge(x,k){}
\drawundirectededge(k,y){} \drawundirectededge(g,y){}

\drawundirectededge(z,DD){} \drawundirectededge(DD,GG){}
\drawundirectededge(GG,EE){} \drawundirectededge(EE,z){}

\drawundirectededge(GG,HH){} \drawundirectededge(HH,LL){}
\drawundirectededge(LL,II){} \drawundirectededge(II,GG){}

\drawundirectedloop[b](A){}\drawundirectedloop[l](b){}\drawundirectedloop[l](BB){}
\drawundirectedloop[l](D){}\drawundirectedloop[r](s){}\drawundirectedloop[r](w){}
\drawundirectedloop[r](E){}\drawundirectedloop(c){}\drawundirectedloop[r](x){}
\drawundirectedloop[l](H){}\drawundirectedloop[b](Y){}\drawundirectedloop[l](v){}
\drawundirectedloop[r](L){}\drawundirectedloop(p){}\drawundirectedloop[l](CC){}
\drawundirectedloop[l](V){}\drawundirectedloop[b](o){}\drawundirectedloop[r](FF){}
\drawundirectedloop[r](J){}\drawundirectedloop[l](X){}\drawundirectedloop[l](HH){}
\drawundirectedloop[l](N){}\drawundirectedloop[l](e){}\drawundirectedloop[r](II){}
\drawundirectedloop[l](P){}\drawundirectedloop[r](i){}\drawundirectedloop(NN){}
\drawundirectedloop[r](R){}\drawundirectedloop[l](t){}\drawundirectedloop[r](n){}
\drawundirectedloop[r](T){}\drawundirectedloop[r](y){}
\end{picture}
\end{turn}
\end{center}
In general, it follows from the recursive definition of the
generators, that each $\B_n$ is a cactus, i.e., a union of cycles
(all of them are of length power of 2) arranged in a tree-like
way. The maximal length of a cycle in $\B_n$ is
$2^{\frac{n+1}{2}}$ if $n$ is odd and $2^{\frac{n}{2}}$ if $n$ is
even. Moreover, for each $n\geq 2$, the graph $\B_n$ contains
exactly $2^{n-1}$ loops rooted at the vertices corresponding to
words in the alphabet $\{0,1\}$ starting by $1$, since the action
of the generator $a$ on these words is trivial.

\begin{prop}\label{numbercycles}
For any $n\geq 4$, the number $c_{n,i}$ of cycles of length $2^i$
in $\B_n$ is:
$$
c_{n,i}=
\begin{cases}
3\cdot 2^{n-2i-1} & \text{for }1\leq i \leq \frac{n}{2}-1\\
3 & \text{for } i = \frac{n}{2}
\end{cases}, \qquad \mbox{for } n \mbox{ even},
$$
and
$$
c_{n,i} = \begin{cases}
3\cdot 2^{n-2i-1} & \text{for }1\leq i \leq \frac{n-1}{2}-1\\
4 & \text{for } i = \frac{n-1}{2}\\
1 & \text{for } i= \frac{n+1}{2}
\end{cases}, \qquad \mbox{for }  n \mbox{ odd}.
$$
\end{prop}

\begin{proof}
It follows from \cite[Proposition 2.2]{noiising}.
\end{proof}

To sum up, we have
$$
|V(\B_n)| = 2^n \qquad \qquad |E(\B_n)|= 2^{n+1}.
$$
As regards edges, note that there are $2^{n-1}$ loops, for $n\geq
2$, and 2 loops in $\B_1$.

Since many computation are trivial for graphs with loops, it is
convenient to define $\B_n^\ast$ as the graph $\B_n$ considered
without loops. Thus, in this case, we have
$$
|V(\B_n^\ast)| = 2^n \qquad \qquad |E(\B_n^ \ast)|=3\cdot2^{n-1}.
$$
For every $n \geq 1 $, denote by $T_n(x,y)$ the Tutte polynomial
$T(\B_n;x,y)$ of $\B_n$ and  by $T_n^\ast(x,y)$ the Tutte
polynomial $T(\B_n^\ast;x,y)$ of $\B_n^\ast$.

\begin{teo}\label{thmtuttebasilica}
For $n\geq 4$, the Tutte polynomial of the Schreier graph $\B_n$
of the Basilica group is
\begin{eqnarray}\label{basilicaodd}
T_n(x,y) &=& y^{2^{n-1}}(y+x+\cdots
+x^{2^{\frac{n-1}{2}}-1})^4(y+x+\cdots
+x^{2^{\frac{n+1}{2}}-1})\\ &\cdot&
\prod_{i=1}^{\frac{n-1}{2}-1}(y+x+\cdots +
x^{2^{i}-1})^{3\cdot 2^{n-2i-1}} \nonumber
\end{eqnarray}
for $n$ odd and
\begin{eqnarray}\label{basilicaeven}
T_n(x,y) =y^{2^{n-1}}(y+x+\cdots
+x^{2^{\frac{n}{2}}-1})^3\prod_{i=1}^{\frac{n}{2}-1}(y+x+\cdots +
x^{2^{i}-1})^{3\cdot 2^{n-2i-1}}
\end{eqnarray}
for $n$ even. Moreover, one has
$$
T_1(x,y)=y^2(x+y) \qquad T_2(x,y)=y^2(x+y)^3 \qquad
T_3(x,y)=y^4(x+y)^4(y+x+x^2+x^3).
$$
\end{teo}

\begin{proof}
The proof follows from Proposition \ref{numbercycles} and Lemma
\ref{singlecycle}. More precisely, by multiplicative property \eqref{prodotto},
$T_n(x,y)$ is obtained as the product of the Tutte polynomials of
its cycles.
\end{proof}

\begin{cor}\label{cor.tutte BAS star}
For each $n \geq 4$, the Tutte polynomial of the graph $\B_n^\ast$  is
$$
T_n^\ast(x,y) = (y+x+\cdots
+x^{2^{\frac{n-1}{2}}-1})^4(y+x+\cdots +x^{2^{\frac{n+1}{2}}-1})\prod_{i=1}^{\frac{n-1}{2}-1}(y+x+\cdots +
x^{2^{i}-1})^{3\cdot 2^{n-2i-1}}
$$
for $n$ odd and
$$
T_n^\ast(x,y) =(y+x+\cdots
+x^{2^{\frac{n}{2}}-1})^3\prod_{i=1}^{\frac{n}{2}-1}(y+x+\cdots +
x^{2^{i}-1})^{3\cdot 2^{n-2i-1}}
$$
for $n$ even. Moreover, one has
$$
T_1^\ast(x,y)= (x+y) \qquad T_2^\ast(x,y)= (x+y)^3 \qquad
T_3^\ast(x,y)= (x+y)^4(y+x+x^2+x^3).
$$
\end{cor}

Let us start by writing the reliability polynomial $R(\B_n,p)$.
\begin{prop}\label{prop.reliability BAS}
For each $n\geq 4$, the reliability polynomial $R(\B_n,p)$ is
given by
$$
R(\B_n,p) = p^{2^n-1}  (1-p)^{2^{n-1}+1}
 \left(
2^{\frac{n-1}{2}}+ \frac{p}{1-p} \right)^4  \left(
2^{\frac{n+1}{2}}+ \frac{p}{1-p} \right)\prod_{i=1}^{\frac{n-1}{2}-1}\left(
\frac{p}{1-p}+2^i\right)^{3\cdot 2^{n-2i-1}}
$$
for $n$ odd and
$$
R(\B_n,p)= p^{2^n-1} (1-p)^{2^{n-1}+1}
 \left( 2^{\frac{n}{2}}+ \frac{p}{1-p} \right)^3\prod_{i=1}^{\frac{n}{2}-1}\left( \frac{p}{1-p}+2^i\right)^{3\cdot
2^{n-2i-1}}
$$
for $n$ even. Moreover, one has
$$
R(\B_1,p)= p(2-p) \qquad R(\B_2,p)=  p^3(2-p)^3 \qquad R(\B_3,p)=
p^7(2-p)^4(4-3p).
$$
\end{prop}
\begin{proof}
It suffices to apply Equation (1) of Theorem \ref{twopolynomials}.
\end{proof}

\begin{os}\rm
Note that the existence of loops does not change the reliability polynomial; therefore  $R(\B_n,p)=R(\B_n^\ast,p)$, as one can directly check.
\end{os}

Evaluating $T_n(x,y)$ in $(1,1)$, we get the
complexity $\tau(\B_n)$, i.e., the number of spanning trees of
$\B_n$.
\begin{prop}\label{prop.complexity BAS}
The complexity of $\B_n$ is
$$
\tau(\B_n) =
\begin{cases}
 2^{\frac{2^{n+2}+3n-5}{6}}& \text{for $n$ odd}\\
2^{\frac{2^{n+2}+3n-4}{6}} & \text{for $n$ even }
\end{cases}.
$$
\end{prop}
\begin{proof}
It suffices to apply Formula (1) of Theorem \ref{evaluations}.
Indeed, one gets
$$
T_n(1,1)=
  2^{\frac{4(n-1)}{2}} \cdot  2^{\frac{n+1}{2}} \prod_{i=1}^{\frac{n-1}{2}-1}2^{i\cdot 3\cdot 2^{n-2i-1}}=
2^{\frac{2^{n+2}+3n-5}{6}}
$$
for $n$ odd and
$$
T_n(1,1)= 2^{\frac{3n}{2}}\prod_{i=1}^{\frac{n}{2}-1}2^{i\cdot 3\cdot 2^{n-2i-1}}= 2^{\frac{2^{n+2}+3n-4}{6}}
$$
for $n$ even. For $n=1,2,3$, one can directly find $T_1(1,1)=2$,
$T_2(1,1)=2^3$, $T_3(1,1)=2^6$.
\end{proof}
\begin{os}\rm
 The previous equations can be
motivated in the following way. In order to have a spanning tree
of $\B_n$, we do not have to consider loops; moreover, we have to
delete exactly one edge in every cycle, so that every cycle of
length $2^i$ contributes by a factor $2^i$.
Since   the loops do not contribute
to the number of spanning trees of $\B_n$, note that   $\tau(\B_n) = \tau(\B_n^\ast)$, for each
$n\geq 1$.
\end{os}

\begin{cor}
The asymptotic growth constant of the spanning trees of $\B_n$ is
$\frac{2}{3}\log 2$.
\end{cor}
\begin{proof}
It suffices to compute
$$
\lim_{n\to \infty}\frac{\log(\tau(\B_n))}{|V(\B_n)|},
$$
with $|V(\B_n)|=2^n$.
\end{proof}

The evaluation of $T_n(x,y)$ in $(1,2)$ provides the number of connected
spanning subgraphs of $\B_n$. More precisely, the following
proposition holds.
 \begin{prop}\label{prop.conn.spanning.subgr BAS}
The number of connected spanning subgraphs of $\B_n$ is
$$
T_n(1,2)= 2^{2^{n-1}}
  \left( 1 + 2^{\frac{n-1}{2}} \right)^4  \left( 1+ 2^{\frac{n+1}{2}}\right)\prod_{i=1}^{\frac{n-1}{2}-1}\left( 1+2^i \right)^{3\cdot 2^{n-2i-1}}
$$
for $n$ odd and
$$
T_n(1,2)= 2^{2^{n-1}}
  \left( 1+ 2^{\frac{n}{2}} \right)^3 \prod_{i=1}^{\frac{n}{2}-1}\left( 1+2^i\right)^{3\cdot 2^{n-2i-1}}
$$
for $n$ even. Moreover, one has
$$
T_1(1,2)= 2^2 \cdot 3 \qquad T_2(1,2)= 2^2\cdot 3^3 \qquad
T_3(1,2)= 2^4\cdot 3^4\cdot 5.
$$
\end{prop}

\begin{proof}
It suffices to apply Formula (2) of Theorem \ref{evaluations}.
\end{proof}

\begin{os}\rm
The value that we have found for the number of connected spanning
subgraphs of $\B_n$ has the following interpretation: the factor
$2^{2^{n-1}}$ corresponds to the possibility of choosing loops in
the subgraph. On the other hand, in order to get a connected
spanning subgraph, each cycle of length $2^i$ contributes by a
factor $2^i+1$, since we can take the whole cycle or delete
exactly one edge from it.
\end{os}

Another interesting computation concerns the number of spanning
forests of the Schreier graph $\B_n$, which is given by
$T_n(2,1)$.
\begin{prop}\label{prop.spanning forest BAS}
The number of spanning forests of $\B_n$ is
$$
T_n(2,1)=   \left( 2^{2^{\frac{n-1}{2}}} -1 \right)^4   \left(  2^{2^{\frac{n+1}{2}}} -1 \right) \prod_{i=1}^{\frac{n-1}{2}-1}\left( 2^{2^i} -1 \right)^{3\cdot 2^{n-2i-1}}
$$
for $n$ odd and
$$
T_n(2,1)=  \left(  2^{2^{\frac{n}{2}}} -1 \right)^3 \prod_{i=1}^{\frac{n}{2}-1}
\left( 2^{2^i} -1 \right)^{3\cdot 2^{n-2i-1}}
$$
for $n$ even. Moreover, one has
$$
T_1(2,1)=  3 \qquad T_2(2,1)=   3^3 \qquad
T_3(2,1)=   3^5\cdot 5.
$$
\end{prop}

\begin{proof}
It suffices to apply Formula (3) of Theorem \ref{evaluations}.
\end{proof}

\begin{os}\rm
The value that we have found for the number of spanning forests of
$\B_n$ has the following interpretation: a spanning forest of
$\B_n$ cannot contain loops nor the whole cycles, since this would
produce a cycle. Therefore, we can choose or not each edge of any
$2^i$-cycle, but we cannot choose in a spanning forest all the
edges of the cycle. Hence, a $2^i$-cycle contributes to the number
of spanning forests by a factor $2^{2^i}-1$. Moreover, observe
that $T_n(2,1)= T_n^\ast(2,1)$, for each $n\geq 1$.
\end{os}

Next, we explicitly verify that by evaluating the Tutte polynomial
of $\B_n$ in $(2,2)$ one gets $2^{|E(\B_n)|}$ (see Formula (4) of
 Theorem \ref{evaluations}).

\begin{prop}
For each $n\geq 1$, one has $T_n(2,2)= 2^{|E(\B_{n})|}=2^{
2^{n+1}}$.
\end{prop}

\begin{proof}
By replacing $x=y=2$ in Equations \eqref{basilicaodd} and
\eqref{basilicaeven}, it turns out that each cycle of length $2^i$
contributes by a factor $2^{2^i}$ to $T_n(2,2)$.
\end{proof}

Finally, by evaluating the Tutte polynomial of $\B_n$ in $(2,0)$,
we investigate the number of acyclic orientations of $\B_n$.
Observe that, whenever we have loops, the number of possible
acyclic orientations on the graphs is $0$. Therefore, we consider
the graph $\B_n^\ast$.

\begin{prop}\label{prop.acyclic orienta BAS}
The number of acyclic orientations on $\B_n^{\ast}$ is
$$
T_n^\ast(2,0)=   \left( 2^{2^{\frac{n-1}{2}}}
-2 \right)^4   \left(  2^{2^{\frac{n+1}{2}}} -2 \right) \prod_{i=1}^{\frac{n-1}{2}-1}\left( 2^{2^i}
-2\right)^{3\cdot 2^{n-2i-1}}
$$
for $n$ odd and
$$
T_n^\ast(2,0)=  \left(  2^{2^{\frac{n}{2}}} -2 \right)^3 \prod_{i=1}^{\frac{n}{2}-1}
\left( 2^{2^i} -2 \right)^{3\cdot 2^{n-2i-1}}
$$
for $n$ even. Moreover, one has
$$
T_1^\ast(2,0)=  2 \qquad T_2^\ast(2,0)=   2^3 \qquad
T_3^\ast(2,0)=   2^5\cdot 7.
$$
\end{prop}

\begin{proof}
By replacing $x=2$ and $y=0$ in Equations \eqref{basilicaodd} and
\eqref{basilicaeven}, it turns out that each cycle of length $2^i$
contributes by a factor $2^{2^i}-2$ to $T_n^\ast(2,0)$.
\end{proof}

\begin{os}\rm
The fact that each cycle of length $2^i$ contributes by the factor
$2^{2^i}-2$ has the following interpretation: there are two
possible orientations for each edge, but we have to avoid the two
cases where we create an oriented cycle (represented in the
following picture in the case of a $2^3$-cycle).

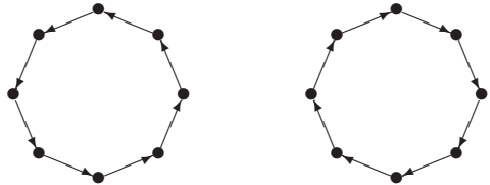
\begin{center}
\begin{picture}(300,95)
\letvertex A=(80,85)\letvertex B=(52,73)\letvertex C=(40,45)\letvertex D=(52,17)
\letvertex E=(80,5)\letvertex F=(108,17)\letvertex G=(120,45)\letvertex H=(108,73)

\drawedge(A,B){} \drawedge(B,C){} \drawedge(C,D){}
\drawedge(D,E){}\drawedge(E,F){} \drawedge(F,G){} \drawedge(G,H){}
\drawedge(H,A){}

\drawvertex(A){$\bullet$}\drawvertex(B){$\bullet$}
\drawvertex(C){$\bullet$}\drawvertex(D){$\bullet$}
\drawvertex(E){$\bullet$}\drawvertex(F){$\bullet$}
\drawvertex(G){$\bullet$}\drawvertex(H){$\bullet$}

\letvertex a=(220,85)\letvertex b=(192,73)\letvertex c=(180,45)\letvertex d=(192,17)
\letvertex e=(220,5)\letvertex f=(248,17)\letvertex g=(260,45)\letvertex h=(248,73)

\drawedge(b,a){} \drawedge(c,b){} \drawedge(d,c){}
\drawedge(e,d){}\drawedge(f,e){} \drawedge(g,f){} \drawedge(h,g){}
\drawedge(a,h){}

\drawvertex(a){$\bullet$}\drawvertex(b){$\bullet$}
\drawvertex(c){$\bullet$}\drawvertex(d){$\bullet$}
\drawvertex(e){$\bullet$}\drawvertex(f){$\bullet$}
\drawvertex(g){$\bullet$}\drawvertex(h){$\bullet$}
\end{picture}
\end{center}
\end{os}

Since $\B_n$ has loops, it does not admit any proper coloring, so
that we investigate the chromatic polynomial of $\B_n^{\ast}$.

\begin{prop}\label{prop.chromatic polyno BAS}
For each $n\geq 4$, the chromatic polynomial $\chi_n(\lambda)$ of
$\B_n^{\ast}$ is
$$
\chi_n(\lambda)\! = -\lambda\!\!
\left( \frac{1-\lambda}{\lambda}  -
\frac{(1-\lambda)^{2^{\frac{n-1}{2}}}}{\lambda} \right)^4 \!\!
\left( \frac{1-\lambda}{\lambda} -
\frac{(1-\lambda)^{2^{\frac{n+1}{2}}}}{\lambda}  \right)\!\! \! \prod_{i=1}^{\frac{n-1}{2}-1}\!\!\!
\left(\frac{1-\lambda}{\lambda}   -
\frac{(1-\lambda)^{2i}}{\lambda} \right)^{3\cdot 2^{n-2i-1}}
$$
for $n$ odd and
$$
\chi_n(\lambda)= -\lambda
\left( \frac{1-\lambda}{\lambda}  -
\frac{(1-\lambda)^{2^{\frac{n}{2}}}}{\lambda} \right)^3 \ \prod_{i=1}^{\frac{n}{2}-1}
\left(\frac{1-\lambda}{\lambda}   -
\frac{(1-\lambda)^{2i}}{\lambda} \right)^{3\cdot 2^{n-2i-1}}
$$
for $n$ even. Moreover, one has
$$
\chi_1(\lambda)=  -\lambda(1-\lambda) \qquad \chi_2(\lambda)=
-\lambda(1-\lambda)^3 \qquad \chi_3(\lambda)=
-\lambda(1-\lambda)^5\cdot (\lambda^2 -3\lambda+3).
$$
\end{prop}

\begin{proof}
It suffices to apply Equation (2) of Theorem \ref{twopolynomials}.
\end{proof}

\begin{os}\rm
Note that $\chi_n(2)=2$, for each $n\geq 1$, according to the fact
that the graph is bipartite and so uniquely $2$-colorable.
\end{os}

Finally, we investigate the relationship between the evaluation of
the Tutte polynomial of the Schreier graph $\B_n^{\ast}$ on the
hyperbola $(x-1)(y-1)=2$ and the partition function of the Ising
model on the same graph. In \cite[Theorem 2.4]{noiising}, the
partition function of the Ising model on $\B_n^{\ast}$ has been
described as
$$
Z_n=2^{2^n} (\cosh(\beta J))^{3\cdot
2^{n-1}}\Phi_n(\tanh(\beta J)),
$$
where $\Phi_n(z)$ is the generating function of closed polygons
for $\B_n^\ast$ given by
$$
\Phi_n(z) =  \left(1+z^{2^{\frac{n-1}{2}}}\right)^4
\left(1+z^{2^{\frac{n+1}{2}}}\right)
\prod_{i=1}^{\frac{n-1}{2}-1}\left(1+z^{2^i}\right)^{3\cdot
2^{n-2i-1}}
$$
for $n\geq 5$ odd and
$$
\Phi_n(z) = \left(1+z^{2^{\frac{n}{2}}}\right)^3 \prod_{i=1}^{\frac{n}{2}-1}\left(1+z^{2^i}\right)^{3\cdot
2^{n-2i-1}}
$$
for $n\geq 4$ even. Moreover,
$$
Z_1=2^2\cosh^2(\beta J)\left(1+\tanh^2(\beta J)\right)
$$
$$
Z_2=2^4\cosh^6(\beta J)\left(1+\tanh^2(\beta J)\right)^3
$$
$$
Z_3=2^8\cosh^{12}(\beta J)\left(1+\tanh^2(\beta
J)\right)^4\left(1+\tanh^4(\beta J)\right).
$$

\begin{teo}\label{thmisingBas}
For each $n\geq 1$, one has
\begin{eqnarray}\label{EQUAisingBASI}
2(e^{2\beta J}-1)^{|V(\B_n^{\ast})|-1}\cdot e^{-\beta
J|E(\B_n^{\ast})|}\cdot T_n^{\ast}\left(\frac{e^{2\beta J}+1}{e^{2\beta
J}-1},e^{2\beta J}\right) = Z_n.
\end{eqnarray}
\end{teo}

\begin{proof}
Here we only prove the case of $n$ even (the computations for $n$
odd are similar). Recall that $|E(\B_n^{\ast})| = 3\cdot 2^{n-1}$
and $|V(\B_n^{\ast})|= 2^n$. Let $e^{\beta J}=t$, so that Equation
\eqref{EQUAisingBASI} can be written as
\begin{eqnarray}\label{equalityisingBAS}
\frac{2(t^2\!-\!1)^{2^n-1}}{t^{3\cdot
2^{n-1}}}\cdot T_n^{\ast}\left(\frac{t^2\!+\!1}{t^2\!-\!1},t^2\!\right)\!\!\!&=&\!\!\!
2^{2^n}\!\left(\frac{t^2\!+\!1}{2t}\right)^{3\cdot 2^{n-1}}
\left(1\!+\!z^{2^{\frac{n}{2}}}\right)^3\!
\prod_{i=1}^{\frac{n}{2}-1}\!\!\left(1\!+\!z^{2^i}\right)^{3\cdot
2^{n-2i-1}}\!\!
\left|_{z=\frac{t^2-1}{t^2+1}}\right..
\end{eqnarray}
One can directly check that
$$
T_n^{\ast}\left(\frac{t^2\!+\!1}{t^2\!-\!1},t^2\right)=
\!\left(\frac{t^2\!-\!1}{2}\!\left(1\!+\!\left(\frac{t^2\!+\!1}{t^2\!-\!1}
\right)^{2^\frac{n}{2}}\right)\!\!\right)^3 \prod_{i=1}^{\frac{n}{2}-1}\left(\frac{t^2\!-\!1}{2}\! \left(1\!+\!\left(\frac{t^2\!+\!1}{t^2\!-\!1}\right)^{2^i}\right) \!\!\right)^{3\cdot
2^{n-2i-1}}
$$
Then, it is not difficult to prove that both sides of Equation
\eqref{equalityising} are equal to
$$
\frac{(t^2+1)^{3\cdot 2^{n-1}}}{t^{3\cdot
2^{n-1}}2^{2^{n-1}}}\!\left(1\!+\!\left(\frac{t^2\!-\!1}{t^2\!+\!1} \right)^{2^\frac{n}{2}}\right)^3\prod_{i=1}^{\frac{n}{2}-1} \left(1\!+\!\left(\frac{t^2\!-\!1}{t^2\!+\!1}\right)^{2^i}\right)^{3\cdot
2^{n-2i-1}}.
$$
\end{proof}

\end{document}